\documentclass[letterpaper, 10 pt, conference]{ieeeconf}
\IEEEoverridecommandlockouts
\usepackage{cite}
\usepackage{amsmath,amssymb,amsfonts}
\usepackage{algorithmic}
\usepackage{graphicx}
\usepackage{algorithm,algorithmic}
\usepackage{hyperref}
\hypersetup{hidelinks=true}
\usepackage{textcomp}
\newtheorem{lemma}{Lemma}
\newtheorem{definition}{Definition}
\newtheorem{theorem}{Theorem}
\newtheorem{proposition}{Proposition}
\newtheorem{assumption}{Assumption}
\newtheorem{remark}{Remark}

\newtheorem{corollary}{Corollary}
\usepackage{xcolor}
\usepackage{setspace}
\newcommand{\PLI}{P\L{}I}
\newcommand{\gPLI}{\textit{g}\PLI}

\newcommand{\kinfPLI}{$\mathcal{K}_\infty$-\PLI}

\newcommand{\SSpc}{X}
\newcommand{\re}[1][\empty]{\mathbb{R}^{#1}}
\newcommand{\loss}{f}
\newcommand{\mloss}{\loss^{\ast}}
\newcommand{\Z}{\mathcal{Z}}
\newcommand{\T}{\mathcal{T}}

\usepackage[normalem]{ulem} 
\newcommand{\MS}[1]{\textcolor{red}{#1}}

\usepackage{cuted} 
\def\BibTeX{{\rm B\kern-.05em{\sc i\kern-.025em b}\kern-.08em
    T\kern-.1667em\lower.7ex\hbox{E}\kern-.125emX}}
\begin{document}
\title{\LARGE \bf On incremental and semi-global exponential stability of gradient flows satisfying generalized Łojasiewicz inequalities}
\author{Andreas Oliveira, \IEEEmembership{Member, IEEE}, Arthur C. B. de Oliveira, \IEEEmembership{Member, IEEE}, \\ Mario Sznaier, \IEEEmembership{Fellow, IEEE} and Eduardo Sontag, \IEEEmembership{Fellow, IEEE}
\thanks{This work was partially supported by NSF grant CMMI 2208182, AFOSR grant FA9550-19-1-0005 and  ONR grant N00014-21-1-2431. The authors are with the ECE Department at Northeastern University, Boston, MA 02115. \textit{(Corresponding author: Andreas Oliveira, e-mail: franciscodemelooli.a@northeastern.edu).}}}

\maketitle

\begin{abstract}
The \L{}ojasiewicz inequality characterizes objective-value convergence along gradient flows and, in special cases, yields exponential decay of the cost. However, such results do not directly give rates of convergence in the state. In this paper, we use contraction theory to derive state-space guarantees for gradient systems satisfying generalized \L{}ojasiewicz inequalities. We first show that, when the objective has a unique strongly convex minimizer, the generalized \L{}ojasiewicz inequality implies semi-global exponential stability; on arbitrary compact subsets, this yields exponential stability. We then give two curvature-based sufficient conditions, together with constraints on the \L{}ojasiewicz rate, under which the nonconvex gradient flow is globally incrementally exponentially stable.
\end{abstract}

\begin{keywords}
gradient flows, \L{}ojasiewicz inequality, contractivity
\end{keywords}

\section{Introduction}\label{sec:introduction}

The Łojasiewicz inequality (LI) \cite{lojasiewicz1963propriete} is a powerful tool for relaxing convexity assumptions while proving convergence of gradient descent and gradient-flow dynamics. For real-analytic functions, Łojasiewicz showed that in a neighbourhood of any local minimizer the function satisfies a local Łojasiewicz inequality, establishing good convergence properties for gradient systems. This idea was later generalized via the Kurdyka–Łojasiewicz (KL) framework and used to establish convergence of a broad class of descent methods, including proximal algorithms and forward–backward splitting schemes \cite{attouch2013convergence}. In parallel, the machine learning literature on nonconvex optimization has embraced the Polyak–Łojasiewicz inequality (\PLI) as a tractable condition for establishing global convergence, optimality, and linear convergence rates for first-order methods \cite{karimi2016linear}. 

Even though the \PLI\ provides a natural Lyapunov function and yields exponential convergence of gradient flows in terms of the objective value, it is generally unclear whether this can be extended to convergence of the state to a minimizer. This implication would be immediate if the objective were uniformly comparable to powers of the Euclidean norm. As shown in \cite{karimi2016linear}, one direction can be guaranteed under a global PL condition: the objective upper bounds a quadratic function of the distance to the minimizer. However, the corresponding two-sided comparability is false in general. Still, the one-sided bound is valuable, as it already implies a form of semi-global exponential stability in the state. In this paper, we leverage contraction theory together with generalized \L{}ojasiewcz-type inequalities to obtain comparable, and in some cases stronger, state-space convergence guarantees.

Contractivity can, in some ways, be viewed as a stronger property than stability. Informally, contractive systems are those for which a suitably defined distance between different trajectories tends to zero exponentially fast, regardless of the initial conditions. Classical notions of contractivity based on logarithmic norms have been widely applied to chemical reaction networks, certain classes of partial differential equations, among others \cite{soderlind2006logarithmic,Sontag-Aminzare}. Despite its broad applicability, we highlight in this work that logarithmic norms are insufficient to capture the contractivity of gradient flows where the cost function satisfies a \PLI.

Instead, we turn to a complementary notion, introduced in \cite{LOHMILLER1998683} based on Riemannian metrics, which studies the exponential stability of the variational dynamics of a nonlinear system and provides conditions under which an appropriately chosen Riemannian distance strictly contracts along the flow. This framework is quite general \cite{GIESLConverse} and has been applied to the analysis of learning algorithms, geodesic convexity, and even data-driven control \cite{kozachkov2022generalizationsupervisedlearningriemannian,Wensing_2020, oliveira2025DDC}.

The main contributions of this work are as follows:

\begin{enumerate}
    \item Inspired by \cite{GIESLConverse}, we first construct a Riemannian metric that establishes global contractivity of the gradient flow of a nonconvex function satisfying the generalized \L{}ojasiewicz-type inequality introduced in \cite{Arthur-PLI}, under the assumption that the function admits a unique locally strongly convex minimizer. This yields semi-global exponential stability of the equilibrium point \(x^\ast\). By restricting the construction to arbitrary compact sets, one obtains a Riemannian metric certifying that \(x^\ast\) is exponentially stable on each such set. Both results go beyond Lyapunov-based certificates, which in general guarantee only convergence in the cost.
   \item The global contractivity result above does not, by itself, imply that the equilibrium point $x^{\ast}$ is globally exponentially stable. To recover global exponential stability of $x^{\ast}$, we derive two sufficient conditions based on the curvature of the cost function and the rate associated with the Łojasiewicz-type inequality. Under these conditions, we construct globally and uniformly bounded Riemannian metrics that contract along the flow. This implies the gradient system is globally incrementally exponentially stable and hence $x^{\ast}$ is globally exponentially stable.
   \item To validate the generalized Łojasiewicz inequality, we exhibit a function that satisfies only a generalized Łojasiewicz condition from \cite{Arthur-PLI}  and whose gradient flow is nevertheless globally contractive.
\end{enumerate}

\section{Preliminaries}

\subsection{Stability Notions}
Consider the nonlinear ordinary differential equation (ODE):
\begin{equation}\label{eq:ode}
    \dot{x} = g(x),
\end{equation}
where $g:\mathbb{R}^n \to \mathbb{R}^n$ is continuously differentiable. We utilize the definitions of (global) exponential stability as given in \cite{khalil2002nonlinear}. We say that an equilibrium $x^{\ast}$ is \emph{semi-globally exponentially stable} (SGES) if there exists a continuous function $h:\mathbb{R}^n \rightarrow \mathbb{R}^{+}$, and rate $\kappa >0$ such that for any initial condition $x_0 \in \mathbb{R}^n$ the corresponding solution, denoted by $\varphi(t,x_0)$, satisfies:
\begin{equation}
    \|\varphi(t,x_0) - x^\ast\| \;\leq\; h(x_0) e^{-\kappa t},
    \qquad \forall t \geq 0.
\end{equation}
\begin{definition}
System \eqref{eq:ode} is said to be \emph{incrementally exponentially stable} (IES) on a forward invariant set $K \subseteq \mathbb{R}^n$ if there exists a rate $\kappa$ and overshoot $c>0$ such that for any two initial conditions $x_0,y_0 \in K$ the corresponding solutions $\varphi(t,x_0)$, $\varphi(t,y_0)$ of \eqref{eq:ode} satisfy:
\begin{equation}
\label{eq:ies}
    \|\varphi(t,x_0)-\varphi(t,y_0)\|
    \;\le\; c\,e^{-\kappa t}\,\|x_0-y_0\|,
    \, \forall\, t\ge 0.
\end{equation}
\end{definition}

Incremental exponential stability of \eqref{eq:ode} implies the existence of a global exponentially stable equilibrium point, but the converse is false.

\subsection{Gradient flows and the generalized Polyak-\L{}ojasiewicz}

Let $\loss:\SSpc\to\re$ be a proper, twice continuously differentiable function. Assume $\loss$ is bounded below and that it attains its minimum at some point in $\SSpc$, and consider the following optimization problem
\begin{equation}
        \label{eq:def-optprob}
        \underset{x\in\SSpc}{\textrm{minimize}}\quad\quad f(x),
\end{equation}
to be solved by finding the equilibrium points of the gradient flow 
\begin{equation}\label{eq:gradient-flow}
    \dot x =-\nabla f(x).
\end{equation}
It is well known that if $\loss$ is proper with a well-defined minimum value, then trajectories of the gradient flow will be pre-compact (see \cite{Arthur-PLI} for a brief overview). If we further assume that $\loss$ is real-analytic, then by \L{}ojasiewicz Theorem we know that every trajectory converges to a unique critical point of $f$ and that trajectories have finite length.

Although these results guarantee convergence of trajectories to the set of critical points, $\Z:=\{x\in\SSpc~|~\nabla\loss(x)=0\}$, there is no guarantee that they will converge to the set of global minimizers of $\loss$, $\T:=\{x\in\SSpc~|~f(x)=\min_{\SSpc}f\}$. A common assumption in the nonconvex optimization literature is to require $\loss$ to satisfy a Polyak-\L{}ojasiewicz inequality, with the recent works \cite{Arthur-PLI,cui2025small,cui2024small} proposing a generalized framework for the standard inequality as follows\MS{:}
\begin{definition}[Generalized Polyak-\L{}ojasiewicz inequality]
    \label{def:genPLI}
    A function $\loss$ is said to satisfy a \emph{generalized} Polyak-\L{}ojasiewicz inequality if there exists some continuous function $\alpha:\re[+]\to\re[+]$ satisfying $\alpha(r)=0\iff r=0$ for which 
    \begin{equation}
        \label{eq:def-genPLI}
        \|\nabla \loss(x)\|\geq\alpha(\loss(x)-\mloss)
    \end{equation}
    holds for all $x\in\SSpc$ with $\mloss:=\min_{x\in\SSpc}\loss(x)$.
\end{definition}

Further assumptions can be imposed on $\alpha$ to characterize specific robustness and convergence guarantees. For example, if $\loss$ satisfies a generalized \PLI\ with 
\begin{equation}
    \label{eq:def-gPLI}
    \alpha(r)=\sqrt{\mu r}
\end{equation}
for some $\mu>0$, we recover the traditional \PLI\ definition, which we will refer to as a \emph{global} \PLI\ (\gPLI). It can be shown that the cost $\loss$ will converge exponentially to its minimum value along a gradient-flow solution if and only if it satisfies a $\gPLI$. Alternatively, it is easy to verify that no further assumption on $\alpha$ is necessary to guarantee asymptotic convergence of $\loss$.

In this paper, we are interested in characterizing under which conditions the gradient flow associated with a loss function $\loss$ is contracting and when that implies the system is IES or merely SGES. To achieve this, we will assume the following of $\alpha$.
\begin{definition}
    \label{def:Kinf}
    A function $\alpha:\re[+]\to\re[+]$ is of class $\mathcal{K}_\infty$ if
    \begin{itemize}
        \item $\alpha(r)=0\iff r=0$;
        \item $\alpha(r_1)<\alpha(r_2)$ for all $r_2>r_1\ge0$;
        \item $\lim_{r\to\infty}\alpha(r)=\infty$.
    \end{itemize}
\end{definition}

If a function $\loss$ satisfies a generalized \PLI\ with $\alpha$ of class $\mathcal{K}_\infty$ we say that it satisfies a \kinfPLI.

\subsection{Contractivity Under Logarithmic Norms}

In this section, we introduce the class of logarithmic norms, define contraction with respect to such functionals, and show that this class of metrics is too restrictive to prove contractivity of gradient flows satisfying a \kinfPLI.

\begin{definition}{\cite{desoer2009feedback}}
Let $\|\cdot\|$ denote a norm on $\mathbb{R}^n$, and let the same symbol denote its induced matrix norm on $\mathbb{R}^{n \times n}$. The \emph{logarithmic norm} of a matrix $A \in \mathbb{R}^{n \times n}$ associated with $\|\cdot\|$ is defined by
\begin{equation}
    \mu(A) := \lim_{h \to 0^+} \frac{\|I + hA\| - 1}{h}.
\end{equation}
\end{definition}

Unless otherwise specified, $\|\cdot\|=\|\cdot\|_2$. We denote by $\mu_p$ the logarithmic norm induced by the $\|\cdot\|_p$, and by $\mu_{p,Q}$ the logarithmic norm induced by the weighted norm $\|x\|_{p,Q}:=\|Qx\|_p$.

Now consider the ODE \eqref{eq:ode}, with Jacobian denoted by $\nabla g := \frac{\partial g}{\partial x}$. The notion of contraction based on logarithmic norms and presented in \cite{contractive-with-inputs-sontag} is defined via the Jacobian $\nabla g$ and presented below:

\begin{definition}{\cite{contractive-with-inputs-sontag}}
The system \eqref{eq:ode} is said to be \emph{infinitesimally contracting} in a set $G \subset \mathbb{R}^n$, if there exists some norm $\|\cdot \|$, with associated logarithmic norm $\mu$,  such that, for some constant $\nu > 0$ (the \emph{contraction rate}), it holds that:
\begin{equation}\label{def:infinitesimally-contracting}
    \mu(\nabla g(x)) \leq -\nu, \quad \forall x \in G.
\end{equation}
\end{definition}

\begin{remark}
    Informally, the key consequence of infinitesimally contracting systems is that if trajectories are forward invariant in a convex set $G$, then the resulting system is IES with overshoot $c = 1$ in the set $G$ \cite{contractive-with-inputs-sontag}.
\end{remark}

Now we introduce an interesting connection between the logarithmic norm, the spectral abscissa and the induced matrix norm.

\begin{lemma}{\cite{desoer2009feedback}}\label{lemma:tightness-sysmmetric}
Let $A \in \mathbb{R}^{n \times n}$ and let $\|\cdot\|$ be any norm with associated logarithmic norm $\mu(\cdot)$. If $\alpha(A) = \max\{\Re(\lambda) : \lambda \text{ is an eigenvalue of } A\}$ denotes the spectral abscissa, then
\begin{equation}
-\|A\| \;\le\; -\mu(-A) \;\le\; \Re(\lambda) \;\le\; \alpha(A) \;\le\; \mu(A) \;\le\; \|A\|.
\end{equation}

Further if $A \in \mathbb{R}^{n \times n}$ is a symmetric matrix then $\mu_2(A) = \lambda_{\max}(A)$.
\end{lemma}

This result is particularly useful in the context of this paper, since for gradient systems, the Jacobian of the vector-field is the negation of the Hessian of the cost, which is always symmetric. This allows one to prove the following statement about contractivity of gradient flows with respect to logarithmic norms:

\begin{proposition}
    The system \eqref{eq:gradient-flow} is infinitesimally contracting in $\mathbb{R}^n$ with contraction rate $\nu$ if and only if $f$ is a strongly convex.
\end{proposition}

\begin{proof}
    First assume the system is infinitesimally contracting with rate $\nu$ for some norm $\|\cdot\|$. By Lemma \ref{lemma:tightness-sysmmetric}, $\alpha(-\nabla ^2 f(x)) \le \mu(-\nabla ^2 f(x)) \le -\nu$. From the definition of $\alpha(\cdot)$ and the previous bound we obtain $\alpha(-\nabla^2 f) = -\lambda_{\min}(\nabla^2 f(x)) \le -\nu.$ Hence by Theorem 2.1.11 in \cite{nesterov2013} this implies $f$ is strongly convex.
    Now assume $f$ is strongly convex, again by Theorem 2.1.11 we obtain $-\lambda_{\min}(\nabla^2 f(x)) \le -\nu \Rightarrow \alpha(-\nabla^2 f(x)) \le -\nu$.
    For the gradient flow system where $f \in C^2$, the hessian $\nabla^2 f := [\frac{\partial^2 f}{\partial x_i \partial x_j}]_{i,j}$ is always symmetric so by Lemma \ref{lemma:tightness-sysmmetric} $\mu_2(-\nabla^2 f) =\alpha(-\nabla^2 f) \le -\nu$ and the system is infinitesimally contracting in the logarithmic norm induced by $\|\cdot\|_2$. 
\end{proof}

\begin{remark}    
    Because a \gPLI\ system can be concave on certain subsets, there can exist points $x$ where $\lambda_{\min}(\nabla^{2} f(x)) < -\beta < 0$. Hence $-\lambda_{\min}(\nabla^{2} f(x)) = \mu_{2}(-\nabla^{2} f(x)) > \beta > 0$, showing that such systems cannot, in general, be \emph{infinitesimally contracting} under logarithmic norms.
\end{remark}

\textbf{Example:}

Consider $f(x) = \frac{1}{2}x^2 - 2 \cos(x)$, this function is verified numerically to satisfy a $\frac{1}{4}-$\gPLI\ condition with unique equilibrium at the origin but $\nabla^2 f(x) = 1 + 2\cos(x) < 0$ in a set of unbounded measure.

\subsection{Contractivity under Riemannian Metrics}

The limitations of logarithmic norms for proving contractivity of gradient flows when the cost satisfies a \gPLI\ motivate us to instead consider the class of Riemannian metrics.

\begin{definition}{\cite{lee2018introduction}}
Let $G$ be an open, path-connected subset of $\mathbb{R}^n$.  A \emph{Riemannian metric} is a function $M \in C^{1}(G,\mathbb{S}^n)$ such that $M(x)$ is positive definite for all $x \in G$.  For $v,w \in \mathbb{R}^n$, the bilinear form $v^T M(x) w$ defines an inner product at $x$.  

We define the orbital derivative of $M$ as:
\begin{equation}
\dot M(x) := \left.\frac{d}{dt} M(\varphi(t,x))\right|_{t=0}.
\end{equation}
If $M \in C^1(G,\mathbb{S}^n)$, then
\begin{equation}
\dot M_{ij}(x) = \nabla M_{ij}(x) \cdot g(x), \quad i,j=1,\dots,n.
\end{equation}
\end{definition}

\begin{definition}
Let $M$ be a Riemannian metric and let
\[
\begin{aligned}
\Gamma(x_0,x_1)
:=\Big\{\gamma:[0,1]\to G \ \text{piecewise }C^1 \ \big|\ 
&\ \gamma(0)=x_0,\\
&\ \gamma(1)=x_1\Big\}.
\end{aligned}
\]
The (Riemannian) distance between $x_0,x_1\in G$ induced by $M$ is
\begin{equation}\label{eq:riem_distance}
\begin{aligned}
d_M(x_0,x_1)
:=\inf_{\gamma\in\Gamma(x_0,x_1)} \int_{0}^{1} (\frac{\partial}{\partial s} \gamma(s)^{\top} M(\gamma(s))\frac{\partial}{\partial s}\gamma(s))^{\frac{1}{2}}ds.
\end{aligned}
\end{equation}
\end{definition}

\begin{definition}{\cite{GIESLConverse}}\label{def:contraction-riemannian}
Let $M$ be a Riemannian metric. For $v \in \mathbb{R}^n$, define
\begin{equation}
L_M(x;v) := \tfrac{1}{2} v^{\top} \big(M(x) \nabla g + \nabla g^{\top} M(x) + \dot M(x)\big) v.
\end{equation}
Then set $\mathcal{L}_M(x) := \max_{v^T M(x) v = 1} L_M(x;v).$
We call $M$ a \emph{contraction metric} on $G$ with rate $-\nu < 0$ if
\begin{equation}\label{eq:contractivity-of-M}
\mathcal{L}_M(x) \le -\nu, \quad \forall x \in G.
\end{equation}

If $G = \mathbb{R}^n$ then we say $M$ is a \emph{global contraction metric}.
\end{definition}

\begin{remark}\label{rm:lyapunov-equation}
Condition \eqref{eq:contractivity-of-M} is equivalent to the matrix inequality,
\begin{equation}
M(x) \nabla g + \nabla g^{\top} M(x) + \dot M(x) \preceq -2\nu M(x), \quad \forall x \in G,
\end{equation}
where $\preceq$ denotes negative semi-definiteness. If $M(x) = Q \in \mathbb{S}^n$ 
by Lemma 2.16 in \cite{FB-CTDS} this is equivalent to $\mu_{2,P}(\nabla g) \leq -\nu$ where $P^2 = Q$.
\end{remark}
\begin{remark}
    Much like the case of contractive systems with respect to logarithmic norms, proving $M$ is a contraction metric on a forward invariant open path-connected set $G$ shows that \cite{aghannan2003intrinsic}: 
    \begin{equation}
    \label{eq:ARestimate}
    \displaystyle d_{M}\big(\varphi(x_0,t),\varphi(x_1,t)\big) \;\le\; e^{-\nu t}\, d_{M}(x_0,x_1).
    \end{equation}
\end{remark}

In general, if $M$ is a global contraction metric, estimate \eqref{eq:ARestimate} does not allow for a direct translation into estimates in the $\ell_2$ norm. The following proposition shows when $\eqref{eq:ARestimate}$ can be translated into $\ell_2$ norm bounds.

\begin{proposition}\label{prop:ies}
    Assume there exists $\alpha,\beta> 0$ such that $\beta I \preceq M(x) \preceq \alpha I$ and $M$ is a global contraction metric, then for any $x_0,x_1 \in \mathbb{R}^n$:
    \begin{equation}\label{eq:ies}
        \|\varphi(x_0,t) - \varphi(x_1,t)\|_2 \leq e^{-\nu t} \big(\frac{\alpha}{\beta}\big)^{\frac{1}{2}} \|x_0-x_1\|_2.
    \end{equation}
\end{proposition}

\begin{proof}
We will first prove the lower bound on $d_{M}\big(x_0,x_1\big)$. From the definition of distance, and by the fact that $v^TM(x)v \geq \beta\|v\|_2^2$, we observe that
\begin{equation}
\begin{aligned}
    & d_{M}\big(x_0,x_1\big) = 
    \inf_{\gamma\in\Gamma(x_0,x_1)} \int_{0}^{1} (\frac{\partial}{\partial s} \gamma(s)^{\top} M(\gamma(s))\frac{\partial}{\partial s}\gamma(s))^{\frac{1}{2}}ds \\
    & \geq \inf_{\gamma\in\Gamma(x_0,x_1)} \beta^{\frac{1}{2}} \int_{0}^1 (\frac{\partial}{\partial s} \gamma(s)^T\frac{\partial}{\partial s}\gamma(s))^{\frac{1}{2}} ds
     = \beta^{\frac{1}{2}} \|x_0 - x_1\|_2.
\end{aligned}
\end{equation}
where the last inequality is obtained by the fact that for $M$ a constant matrix, the geodesic is a straight line. Similarly $d_{M}\big(x_0,x_1\big) \leq \alpha^{\frac{1}{2}} \|x_0 - x_1\|_2$ and combining the two inequalities with \eqref{eq:ARestimate} proves \eqref{eq:ies}.
\end{proof}

Thus, establishing IES for a gradient flow reduces to constructing a global contraction metric $M$ that is \emph{uniformly bounded above and below}. Likewise, SGES follows if one can construct a global contraction metric that is \emph{uniformly bounded below.}

Lastly, we present an important lemma that will be used later to construct our Riemannian metrics. We omit the proof and defer the reader to Lemma 2.7 of \cite{GIESLConverse}.

\begin{lemma}{\cite{GIESLConverse}}\label{lemma:riemmanina-metric-with-lyapunov}
Let $N: G \to S^n$ be a Riemannian metric and $V: G \to \mathbb{R}$ be continuous and orbitally continuously differentiable.  
Then
\begin{equation}
M(x) := e^{2V(x)} N(x)
\end{equation}
is also a Riemannian metric, and
\begin{equation}
\mathcal{L}_M(x) = \mathcal{L}_N(x) + \dot V(x).
\end{equation}
\end{lemma}

\section{\kinfPLI\ Functions and Contraction Metrics}

For this paper, we will consider the optimization problem given in \eqref{eq:def-optprob} where the cost $\loss$ is proper, twice continuously differentiable, has a unique global minimizer at which the function is locally strongly convex, and satisfies a \kinfPLI. In this setting, it is known that the loss converges to the global minimizer exponentially fast along the flow if and only if $\loss$ satisfies a \gPLI. In contrast, if $\loss$ fails to satisfy a \gPLI\ but does satisfy a \kinfPLI, one can typically only guarantee asymptotic convergence of the cost.

The purpose of this section is to go beyond the convergence of the cost and establish convergence guarantees in the state. We begin by showing that if $\loss$ satisfies a \kinfPLI\ and is locally strongly convex around the minimizer, then the convergence of the objective along the flow can always be strengthened to SGES of the minimizer. When the feasible set is additionally compact, we can further show that the minimizer is exponentially stable in said set. Lastly, we provide sufficient conditions, based on \kinfPLI-type upper bounds and curvature assumptions on $\loss$, under which the SGES property of the solution can be upgraded to IES. These results depart from classical converse global contraction-metric theorems \cite{GIESLConverse} and Lyapunov-based approaches \cite{khalil2002nonlinear} in that the contraction metrics we construct do not depend on the unknown solutions of the gradient-flow dynamics, but instead are expressed directly in terms of the cost function.

\subsection{Semi-global exponential stability and Exponential Stability on Compacts}

\begin{assumption}\label{assump:1}
The function $\loss$ admits a unique global minimizer $x^{\ast}$ where it is strongly convex, hence there exists $\nu > 0$ such that $\lambda_{\min}(\nabla^2 f(x^{\ast})) \ge \nu$. Further, $f$ satisfies a \kinfPLI.
\end{assumption}

\begin{theorem}\label{prop:global}
Suppose $\loss$ satisfies Assumption \ref{assump:1}. Then for any $\epsilon \in (0,\nu)$ there exists a uniformly lower bounded Riemannian metric $\hat{M} \in C^{1}(\mathbb{R}^n, \mathbb{S}^{n})$ such that
\begin{equation}\label{eq:globalness}
\begin{aligned}
    \mathcal{L}_{\hat{M}}(x) \;\le\; -\nu + \epsilon, \; \forall x \in \mathbb{R}^n.
\end{aligned}
\end{equation}
Hence, the system is \emph{semi-globally exponentially stable} with rate $\nu-\epsilon$.
\end{theorem}
The proof is provided in the appendix; however, we discuss the idea next. Notice that locally around $x^*$, Lemma \ref{lemma:tightness-sysmmetric} and Remark \ref{rm:lyapunov-equation}, show that the metric $M(x) = I$ suffices to produce $\mathcal{L}_{I}(x) \le -\nu + \epsilon$ for x in a sufficiently small neighborhood of $x^{\ast}$ and in fact it is uniformly negative around any point where $\loss$ is locally strongly convex. Further, inspired by the application of Lemma \ref{lemma:riemmanina-metric-with-lyapunov} in \cite{GIESLConverse}, given that the cost yields a natural Lyapunov function for the gradient flow, we can expect to construct a metric that satisfies \eqref{eq:globalness} globally. One challenge lies in interpolating between a local constant metric $I$, used where $\loss$ is strongly convex, and a state-dependent metric $M(x) = e^{v(f(x))}I$ where $\loss$ isn't. 

Furthermore, notice that the semi-global exponential convergence result in Theorem \ref{prop:global} is purely asymptotic, since the value of the overshoot is unbounded in general. Despite that, we can state the following simple but key result from Theorem \ref{prop:global}:

\begin{corollary}\label{cor:compact}
Suppose $\loss$ satisfies Assumption \ref{assump:1}. Then given any compact, forward invariant neighborhood $K \subset \mathbb{R}^n$ of $x^{\ast}$, and any $\epsilon \in (0,\nu)$, the gradient flow system is \emph{incrementally exponentially stable} with rate $\nu -\epsilon$ in $K$ and $x^{\ast}$ is exponentially stable in $K$.
\end{corollary}

The proof is given in the appendix. The main takeaway, however, is that establishing SGES for gradient systems satisfying Assumption \ref{assump:1} provides an alternative way to show that gradient systems which are locally exponentially stable and globally asymptotically stable are, in fact, exponentially stable on arbitrary compact sets. 

\subsection{Global Incremental Exponential Stability}

We now show how to construct uniformly bounded metrics that guarantee global contractivity under various \kinfPLI\ bounds and different degrees of concavity. For simplicity, we do not require the metric to be $\epsilon$-optimal, although the same argument could be adapted to establish that property as well. The key properties that will allow the two results relate to bounding how concave a function is allowed to be in some sense. We formalize this intuition through two different definitions.

\begin{definition}
    A function $\loss:\re[n]\to\re$ is $m$-\emph{state-bounded concave} for some $m>0$ if the set $S = \{ x \quad | \quad\lambda_{\min}(\nabla^2 f(x)) < m \}$ is bounded.
\end{definition}

\begin{definition}
    A function $\loss:\re[n]\to\re$ is $m$-\emph{magnitude-bounded concave} for some $m>0$ if $\lambda_{\min}(\nabla^2\loss(x))>-m$ for all $x\in\re[n]$.
\end{definition}

With these definitions, we can state the next result:

\begin{theorem}\label{prop:concavity-on-annulus}
    Suppose $f$ satisfies Assumption \ref{assump:1}. Further assume $\loss$ is $m$-state-bounded concave for some $m>0$. Then the gradient flow is \emph{globally incrementally exponentially stable} and the equilibrium $x^{\ast}$ is globally exponentially stable.
\end{theorem}

The proof is provided in the appendix, but the main idea is to construct a globally uniformly bounded contraction metric. The key additional point beyond Theorem~\ref{prop:global} is that the set where $\loss$ fails to be strongly convex is bounded. Hence the contraction metric only needs to be state-dependent on a bounded region; outside it, one smoothly patches it to a constant metric while preserving the contraction inequality. As a result, the metric is uniformly bounded above and below, and Proposition~\ref{prop:ies} yields global IES. Figure~\ref{fig:annulus_concavity} illustrates this regime.

\begin{figure}[t]
    \centering
    \includegraphics[width=\columnwidth]{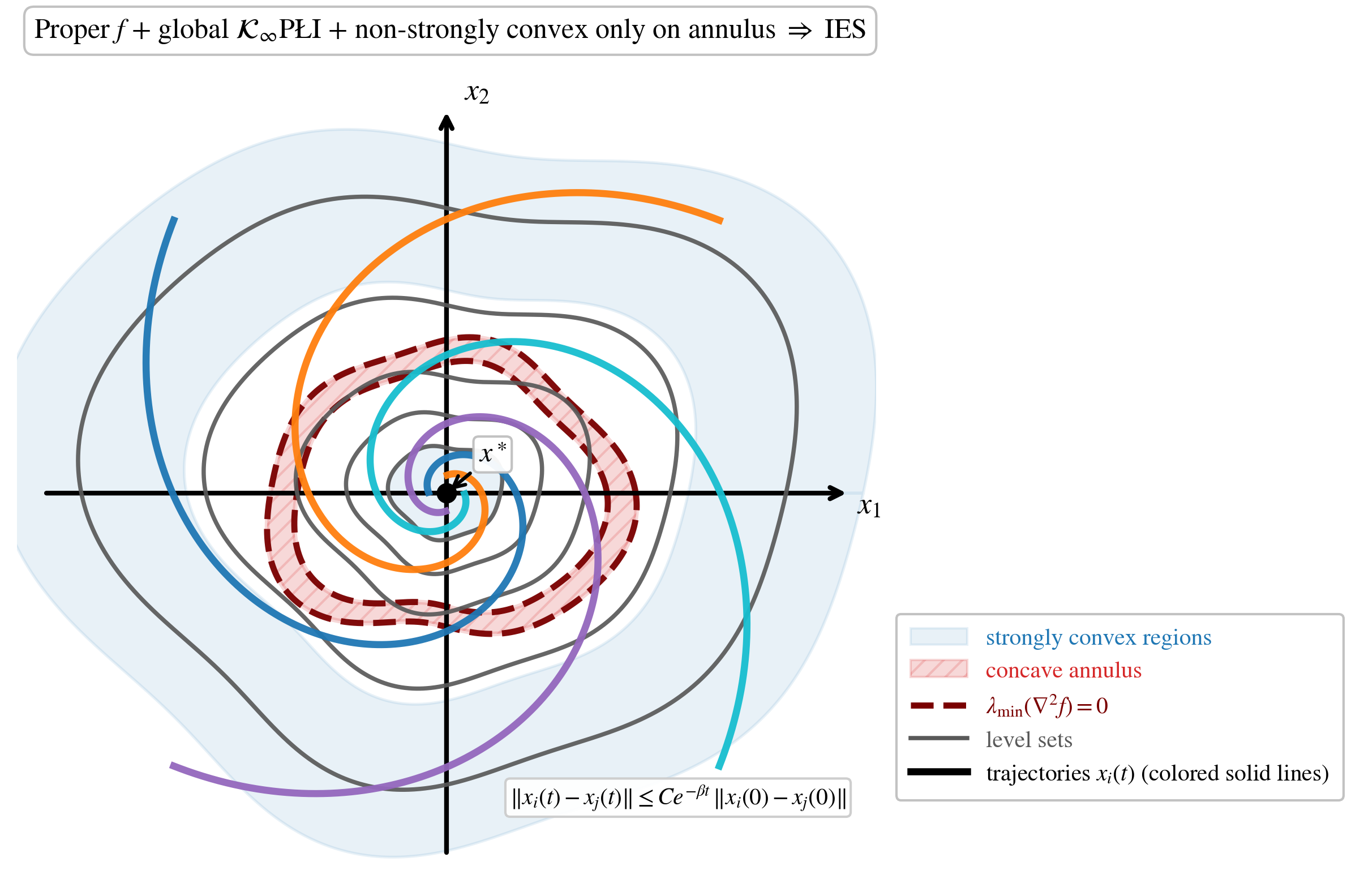}
    \caption{Schematic illustration of the assumptions and conclusion of Theorem~\ref{prop:concavity-on-annulus}. A proper objective $\loss$ with satisfying \kinfPLI. The equilibrium $x^\ast$ lies in a strongly convex neighborhood, and everything outside an annulus is strongly convex. A bounded region where the function is strongly concave is allowed $\lambda_{\min}(\nabla^2 f(x))<0$ and is shown in red. Gradient flow trajectories $\dot{x}=-\nabla f(x)$ starting in different quadrants tend exponentially toward each other since the system is IES by Theorem \ref{prop:concavity-on-annulus}.}
    \label{fig:annulus_concavity}
\end{figure}

Next, we look at magnitude-bounded concave functions, for which the following class of generalized \PLI\ is useful 

\begin{definition}
    A function $\loss$ is said to satisfy a \emph{strong}-\PLI\ if it satisfies \eqref{eq:def-genPLI} for some comparison function $\alpha$ that is of class $\mathcal{K}_\infty$ and for any $s_0 > \mloss$
    \begin{equation}\label{eq:strong-PLI}
        \int_{s_0}^\infty \frac{1}{\alpha(s - \mloss)^{2}} \; ds<\infty.
    \end{equation}
\end{definition}
From this, we present the following result: 

\begin{theorem}\label{prop:above-PL-regime}
Suppose $f$ satisfies Assumption \ref{assump:1}, where the class \kinfPLI\ is in fact a strong-\PLI. Further assume $\loss$ is $m$-magnitude-bounded concave for some $m>0$. Then the gradient flow is \emph{globally incrementally exponentially stable} and the equilibrium $x^{\ast}$ is globally exponentially stable.
\end{theorem}

The proof is given in the appendix and follows the same construction as in Theorem~\ref{prop:global}. The new issue is that $\loss$ may remain concave on unbounded sets, so the metric must stay state-dependent at infinity. The strong-\PLI\ assumption is used to control this growth and ensure that the resulting metric remains uniformly bounded. Proposition~\ref{prop:ies} then implies global IES. Figure~\ref{fig:prop_integrability_envelope} illustrates this case.

\begin{figure}[t]
    \centering
    \includegraphics[width=\columnwidth]{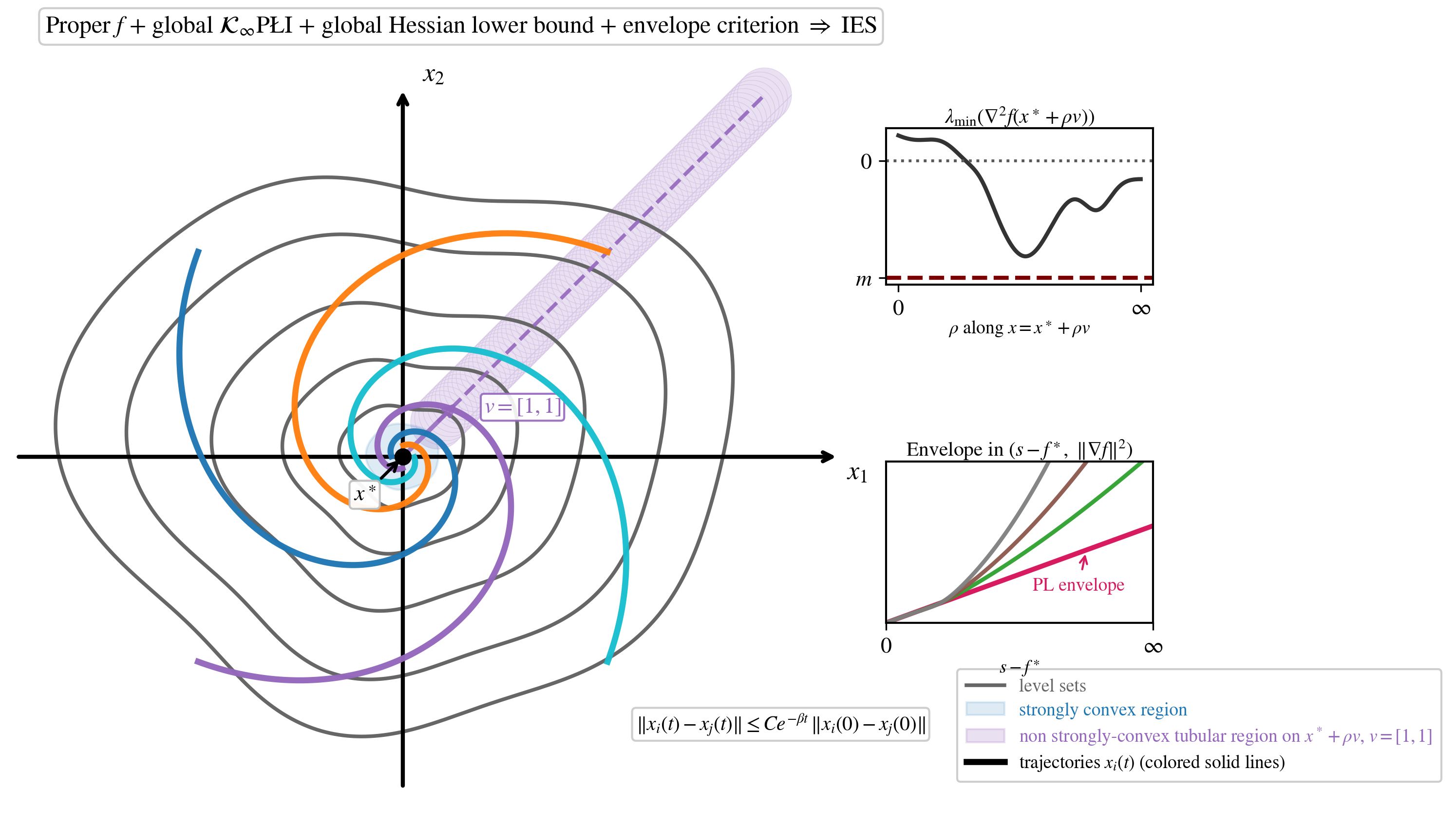}
    \caption{Schematic illustration of the assumptions and conclusion of Theorem \ref{prop:above-PL-regime}. The equilibrium $x^\ast$ lies in a strongly convex neighborhood (shaded blue region). A proper objective $f$ satisfying a \kinfPLI\ that is strictly stronger than the \PLI\ is shown in the lower right graph, and a global Hessian spectral abscissa lower bound (which need not be positive) is shown in the top right graph. Gradient flow trajectories $\dot{x}=-\nabla f(x)$ starting in different quadrants tend exponentially toward each other since the system is IES by Theorem \ref{prop:above-PL-regime}.}
    \label{fig:prop_integrability_envelope}
\end{figure}

\subsection{Discussion}

The existence of a global contraction metric under a \kinfPLI\ bound that cannot be reduced to a \PLI\ type may appear unintuitive. Indeed, one might expect that exponential convergence in some metric should be reflected by the objective value itself, viewed as a natural Lyapunov function. To clarify this apparent mismatch between contraction metrics and Lyapunov decrease in the cost, we point to~\cite{forni2015differential}, which shows that a contraction metric exists on any basin of attraction of a stable hyperbolic equilibrium. 
The key ingredient in that result is the Hartman--Grobman theorem, and the resulting metric is therefore not given by an explicit construction. From this perspective, the present paper can be viewed as a concrete case study that develops constructive routes to global contraction metrics starting from globally asymptotically stable gradient systems.

As noted in the logarithmic norms section, it is generally impossible to prove that a gradient system satisfying a \gPLI\ is infinitesimally contracting. While generally true, the setting of Theorem \ref{prop:concavity-on-annulus} is more structured: the system is infinitesimally contracting everywhere except on a compact annulus around the global minimizer, and all trajectories converge asymptotically to that minimizer. This makes it possible to use logarithmic norms to show that the system is IES even though it is not infinitesimally contracting. One can generalize this beyond gradient systems: any globally asymptotically stable system that is infinitesimally contracting everywhere except on a compact annulus around the equilibrium is IES. We refer the reader to Theorem \ref{thm:ies-inf-contracting-compact} in the Appendix for the formal statement and proof.

\section{Example}
Consider the function:
\begin{equation}
    f(x) = x\log(x + \sqrt{1 + x^2}) - \sqrt{1 + x^2} + 1.
\end{equation} Since $ \log(x + \sqrt{1 + x^2}),\sqrt{1 + x^2}$ are $C^{\infty}$, $f$ is $C^{\infty}$. Computing:

\begin{equation}
    \nabla f(x) = \log(x + \sqrt{1 + x^2}), \quad \nabla^2 f(x) = \frac{1}{\sqrt{1 + x^2}}.
\end{equation}

we observe that global minimizer is at $x = 0$ and the function is locally strongly convex around $0$. Now we find its global \kinfPLI\ bound. Assume $f$ satisfies a \kinfPLI\ with $\alpha(s) = (cs)^{\frac{1}{p}}$, then:

\begin{equation}
    \frac{|\nabla f(x)|^p}{f(x)} = \frac{[\log(x + \sqrt{1 + x^2})]^p}{f(x)} \ge c.
\end{equation}

In the limit as $x \to \infty$, $\nabla f(x) \sim \log(2 x)$ and $f(x) \sim x\log(2x)$, hence we observe that for any $p > 0$, 

\begin{equation}
\lim_{x\to \infty} \frac{\log(x + \sqrt{1 + x^2})^p}{f(x)} \to 0.
\end{equation}

This proves the function $f$ satisfies no global \L{}ojasiewicz inequality, i.e $\alpha(s)$  cannot be of the form $(cs)^{\frac{1}{p}}$. Instead we will show that $f$ satisfies a \kinfPLI\ with  $\alpha(s) = \frac{1}{2}(\log(1 + s))$. First notice that $\log(x + \sqrt{1 + x^2})$ is an odd function  and $f(x)$ is an even function. Hence it suffices to show that for $x > 0$,

\begin{equation}\label{eq:log-K-inf}
    \nabla f(x) \geq \frac{1}{2}\log(1 + f(x)) \Rightarrow e^{2\nabla f(x)} \ge 1 + f(x).
\end{equation}

Using that $\log(x + \sqrt{1 + x^2}) \le x$ for $x > 0$ allows us to obtain that $f(x) \le x^2$ for $x > 0$ and we get that $1 + f(x) \le 1 + x^2$. Furthermore, $e^{2\nabla f(x)} = (x + \sqrt{1 + x^2})^2 \ge 1 + x^2$, which combining with \eqref{eq:log-K-inf} shows that $\alpha(s) = \frac{1}{2}(\log(1 + s))$ is clearly a \kinfPLI\ for the function $f$. Lastly, since $f$ is strongly convex around the minimizer $x^{\ast} = 0$ it satisfies the assumption of Theorem \ref{prop:global} and the gradient flow is contractive under the metric constructed there which further implies $x^{\ast}$ is semi-globally exponentially stable.

\section{Conclusion}

In this paper, we developed techniques for constructing global contraction metrics for gradient systems of nonconvex functions. These metrics provide convergence guarantees in the state, not only in the objective value. Specifically, when the objective is strongly convex near its global minimizer and satisfies a \kinfPLI, the gradient system is semi-globally exponentially stable, and hence exponentially stable on arbitrary compact sets. Under additional assumptions on the \kinfPLI\ and the curvature of the objective, we prove that the gradient system is globally incrementally exponentially stable, and therefore globally exponentially stable.

\bibliographystyle{IEEEtran} 
\begin{spacing}{.80}
\bibliography{ref.bib}{}
\end{spacing}

\section{Appendix}

\subsection{Proof of Theorem \ref{prop:global}}

To prove Theorem \ref{prop:global}, we first state and prove a Lemma used to interpolate between a constant and a state-dependent metric.

\begin{lemma}\label{lemma:partition-of-metrics}
Suppose $\loss$ satisfies Assumption \ref{assump:1}. Define a continuous scalar function $v : \mathbb{R} \to \mathbb{R}$ which is non-decreasing and continuously differentiable on the interval $v^{-1}(a-\epsilon, b + \epsilon)$ for some $a,b,\epsilon$. Define $X = (v \circ f)^{-1}(a,b)$ and assume that for $M(x) = e^{2v(f(x))}I$ (which is in $C^1((v \circ f)^{-1}(a-\epsilon, b + \epsilon), \mathbb{S}^n))$ we have, for the gradient flow \eqref{eq:gradient-flow}, that it satisfies
\begin{equation}
    \mathcal{L}_{M}(x) \;\le\; -\beta < 0, \; \forall x \in X,
\end{equation}
Let $X^{c}$ denote the complement of $X$. Assume further that, for all $x \in X^{c}$, one has that $\lambda_{\min}(\nabla^2 f(x)) \ge \nu > 0$. Then there exists a metric $\hat{M}(x) \in C^1(\mathbb{R}^n,\mathbb{S}^n)$ which satisfies:
\begin{equation}
    \mathcal{L}_{\hat{M}}(x) \;\le\; -\min(\beta,\nu), \; \forall x \in \mathbb{R}^n
\end{equation}
\end{lemma}

\begin{proof}
    Assume first that $X$ is bounded.
    Consider a smooth non-negative function $\sigma : \mathbb{R} \rightarrow [0,1]$ satisfying,
    \begin{equation*}
    \begin{aligned}
        & \sigma(t) = 0 \quad \text{for } t \le 0 \\
        & \sigma(t) = 1  \quad \text{for } t \ge 1,
    \end{aligned}
    \end{equation*}
    and define,
    \begin{equation*}
        g'(s) = \begin{cases}
            0 & s \le  a- \epsilon\\
            \sigma\left(\frac{s - (a - \epsilon)}{\epsilon}\right) & a- \epsilon < s < a \\
            1 & a \le s \le b \\
            1 - \sigma\left(\frac{s - b}{\epsilon} \right) & s > b\\
        \end{cases} \\
    \end{equation*}
    so that the integral function $g(s)$ is constant for $s \le a-\epsilon$ and $s \ge b + \epsilon$ and satisfies $g(s) = s$ for $s \in [a,b]$. Let $\hat{M}(x) = e^{2 g(v(f(x)))}I$. We need to show that $\mathcal{L}_{\hat{M}}(x) \le -\min(\beta,\nu)$ for all $x \in \mathbb{R}^n$.  First notice that  $f$ is $\nu-$strongly convex on $(v \circ f)^{-1}[v(\mloss),a-\epsilon] \subset X^c$ and $g(v(f(x))$ is constant in that set so $\hat{M} = e^{2c}I$ for some constant $c > 0$ and $\mathcal{L}_{\hat{M}}(x) = \mathcal{L}_{I}(x) \le -\nu$. For $x \in (v \circ f)^{-1}(a-\epsilon, a]$ we apply Lemma \ref{lemma:tightness-sysmmetric} and \ref{lemma:riemmanina-metric-with-lyapunov} to get:
    \begin{equation}\label{eq:contraction-with-classK-weighting}
    \begin{aligned}    
        & \mathcal{L}_{\hat{M}}(x) = -\lambda_{\min}(\nabla^2f(x)) - \nabla g(v(f(x)))^{\top} \nabla f(x) \\
        & = -\lambda_{\min}(\nabla^2f(x)) - g'(v(f(x)))v'(f(x)) \|\nabla f(x)\|^2 \\
        & \leq -\lambda_{\min}(\nabla^2f(x)) - g'(v(f(x)))v'(f(x))\alpha(f(x) - \mloss)^2.
    \end{aligned}
    \end{equation}
    For all $x$, $v$ is non-decreasing so $v'(f(x)) \ge 0$ and also by definition $g'(v(f(x))) \ge 0 $. Therefore since $x \in (v \circ f)^{-1}(a-\epsilon, a] \subset X^c$, it still holds that $\mathcal{L}_{\hat{M}}(x) \le -\nu$. For $x \in X$ we have that $g(v(f(x))) = v(f(x))$ hence $\hat{M} = M$ and $\mathcal{L}_{\hat{M}}(x) \le -\beta$. Therefore $\mathcal{L}_{\hat{M}}(x) \le -\min(\beta,\nu)$ for all $x \in (v \circ f)^{-1}[v(\mloss),b)$.
    
    \noindent For $x \in (v \circ f)^{-1}[b, \infty) \subset X^c$ we know that $-\lambda_{\min}(\nabla^2f(x)) \leq -\nu$. Considering first $x \in (v \circ f)^{-1}[b,b+\epsilon]$, we get again by using that $g'(v(f(x))) \ge 0$, $v'(f(x)) \ge 0$ and \eqref{eq:contraction-with-classK-weighting} that $\mathcal{L}_{\hat{M}}(x) \le -\nu$. For $v(f(x)) \ge b + \epsilon$ we have that $g'(v(f(x))) = 0$ so for some $c_2 > 0$,  $g(v(f(x))) = c_2$. Therefore we have that $\hat{M}(x) = e^{2c_2}I$ and $\mathcal{L}_{\hat{M}} \le -\nu$ for $x \in (v \circ f)^{-1}[b+\epsilon,\infty)$. This finalizes the proof that $\mathcal{L}_{\hat{M}}(x) \le -\min(\beta,\nu)$ for all $x \in \mathbb{R}^n$. Since $v\circ f$ is $C^1$ on $(v\circ f)^{-1}((a-\epsilon,b+\epsilon))$, and $g$ is constant outside $(a-\epsilon,b+\epsilon)$ with $g'=0$ at the endpoints, it follows that $\hat{M}\in C^1(\mathbb{R}^n,\mathbb{S}^n)$.
    If $X$ is unbounded, the proof follows similarly, but we instead define
    \begin{equation}
    g'(s) = \begin{cases}
            0 &  s \le a -\epsilon, \\
            \sigma \left(\frac{s - (a-\epsilon)}{\epsilon} \right) & s \ge a -\epsilon. \\
        \end{cases} 
    \end{equation}
\end{proof}

We are now ready to state the proof of Theorem \ref{prop:global}.

\begin{proof}
   First we prove that there exists $M(x)$ such that $\mathcal{L}_{M}(x) \le -\nu$ on $\mathbb{R}^n \setminus G$ where $G$ is a small open connected set around $x^*$. Let $g$ be a real-valued function and define:
   \begin{equation}
       M(x) = Ie^{2g(f(x))}.
   \end{equation}
   We have by Lemma \ref{lemma:tightness-sysmmetric} and Lemma \ref{lemma:riemmanina-metric-with-lyapunov}:
   \begin{equation}
     \mathcal{L}_{M}(x) = -\lambda_{\min}(\nabla^2f(x)) - \frac{\partial g}{\partial f(x)} \|\nabla f(x)\|^2
   \end{equation}
   If $\frac{\partial g}{\partial f} \ge 0$ then by the \kinfPLI\ bound on $\loss$:
   \begin{align*}
       & \mathcal{L}_{M}(x)  \le - \lambda_{\min}(\nabla^2f(x)) - \frac{\partial g}{\partial f(x)} \alpha(f(x)-f^*)^2.
   \end{align*}
    Therefore some algebraic manipulation shows that for $\mathcal{L}_{M}(x) \leq -\nu$, it is sufficient to construct a continuously differentiable $g$ such that:
    \begin{equation}\label{eq:bound-partial-g}
        \frac{\partial g}{\partial f(x)} \geq \max \{\max_{{\{z \in f^{-1}(f(x))\}}} \frac{\nu - \lambda_{\min}(\nabla^2 f(z))}{\alpha(f(x)-f^*)^2},0\}.
    \end{equation}
   Define:
    \begin{equation}
        m(x) = \nu - \lambda_{\min}(\nabla^2 f(x)),
    \end{equation}
    \begin{equation}\label{eq:m*}
        m^*(y) = \max\{m(x): f(x) = y\},
    \end{equation}
    Since $f\in C^2$, $\lambda_{\min}(\nabla^2 f(x))$ is continuous and so is $m(x)$. Because $f$ is proper, $f^{-1}(\cdot)$ is a compact valued map and $m^*(y) = m(x_y)$ for some $x_y \in f^{-1}(y)$. Consider any sequence $y_n \to y$, we can assume without loss of generality that the sequence belongs to some compact set $[y_l, y_u]$, hence there exists a subsequence $x_n$ such that $x_n \to x$ and $m(x_n) = m^*(y_n)$. By continuity of $m(\cdot)$:
    \begin{equation}
    \begin{aligned}
        & \limsup_{y_k \to y} m^*(y_k)  = \lim_{n \to \infty} m^*(y_n) = \\ 
        & \lim_{n \to \infty} m(x_n) = m(x) \le m^*(y)
    \end{aligned}
    \end{equation}
    Therefore $m^*$ is upper semicontinuous and locally bounded above. We can then construct a piecewise linear function $m^u$ which satisfies $m^u(f(x)) > \max\{m^*(f(x)),0\}$. 
    \noindent Now let $G = f^{-1}[\mloss,\mloss+\delta)$ where $\delta$ is sufficiently small so that $\lambda_{\min}(\nabla^2 f(x)) > \nu - \varepsilon$ for all $x \in G$ and define:
    \begin{equation}
        g(s) = \begin{cases}
            \int_{f^{\ast} + \frac{\delta}{2}}^{s}  \frac{m^u(\tau)}{\alpha(\tau - \mloss)^2} d\tau, & s \ge f^{\ast} + \frac{\delta}{2} \\
            0 & s \le f^{\ast} + \frac{\delta}{2}.
        \end{cases}
    \end{equation}
    By construction $g \in C{^1}((f^{\ast} + \frac{\delta}{2},\infty),\mathbb{R})$, is non-decreasing and satisfies \eqref{eq:bound-partial-g} in $\mathbb{R}^n \setminus G$. Therefore $M(x) \in C^{1}(\mathbb{R}^n \setminus G, \mathbb{S}^n)$ and: 
    \begin{equation}
        \mathcal{L}_{M}(x) \;\le\; -\nu, \; \forall x \in \mathbb{R}^n \setminus G.
    \end{equation}
    This allows us to apply Lemma \ref{lemma:partition-of-metrics}, with $X = (g \circ f)^{-1}(g(f^* + \delta), \infty)$, to construct $\hat{M}(x) \in C^{1}(\mathbb{R}^n,\mathbb{S}^n)$ such that:
    \begin{equation}
        \mathcal{L}_{\hat{M}}(x) \;\le\; -\nu +\varepsilon, \; \forall x \in \mathbb{R}^n.
    \end{equation}    
    By construction, $M(x)$ is uniformly lower bounded since $m^u(\cdot) \ge 0$ and from the proof of Lemma \ref{lemma:partition-of-metrics}, so is $\hat{M}(x)$. So from the lower-bound estimate in the proof of Proposition \ref{prop:ies} we can derive that there exists $c > 0$ such that:
    \begin{equation}
        \|\varphi(x_0,t) - x^{\ast}\| \le c d_{\hat{M}}(x_0,x^{\ast})e^{-(\nu - \epsilon) t}
    \end{equation}
    Therefore the equilibrium is semi-globally exponentially stable.
\end{proof}

\subsection{Proof of Corollary \ref{cor:compact}}

\begin{proof}
    From the proof of Theorem \ref{prop:global}, we know that there exists a contraction metric $\hat{M} \in C^1(K,\mathbb{S}^n)$ for any compact neighborhood $K$ around $x^{\ast}$ that satisfies $\mathcal{L}_{\hat{M}}(x) \le -\nu +\epsilon$. Since $K$ is compact $\hat{M}$ is uniformly bounded above and below and by Theorem \ref{prop:ies} we know the gradient flow is IES in $K$ which implies $x^{\ast}$ is exponentially stable in $K$.
\end{proof}

\begin{remark}\label{rm:restricted-form-M}
    We won't prove it here in the interest of space, but for any compact set $K$ containing $x^{\ast}$, we can restrict the metric $\hat{M}$ to be of the form $e^{2cf(x)}I$ for some $c > 0$ to show that:
    \begin{equation}
        \mathcal{L}_{\hat{M}}(x) \;\le\; -\nu + \epsilon, \; \forall x \in K.
    \end{equation}
    The proof follows almost identically to Theorem 4.1 of \cite{GIESLConverse}. We will use this simpler $M$ to simplify the proof of Theorem \ref{prop:concavity-on-annulus}.
\end{remark}

\subsection{Proof of Theorem \ref{prop:concavity-on-annulus}}

\begin{proof}
    By Remark \ref{rm:restricted-form-M} on any compact set $K$, there exists $c > 0$, such that for $M(x) = e^{cf(x)}I$,  $\mathcal{L}_{M}(x) \le -l < 0$ for all $x \in K$. Therefore, since $S$ is bounded, $\loss$ is continuous, proper and strongly convex on $S^c$, there exists $b,c > 0$ such that the compact set $X = (cf)^{-1}([\mloss,b])$ satisfies $X \supset S$ and $\mathcal{L}_{M}(x) \leq -l < 0$ for all $x \in X$. Since $f$ is state-bounded concave on $S$, it holds that for all $x \in X^c$, $\lambda_{\min}(\nabla^2 f(x)) > m > 0$. Therefore, by the proof of Lemma \ref{lemma:partition-of-metrics} there exists a metric $\hat{M}(x) = C^{1}(\mathbb{R}^n,\mathbb{S}^n)$, which is uniformly bounded below and above, satisfying:
    \begin{equation}
        \mathcal{L}_{\hat{M}}(x) \le -\min\{m,l\} < 0, \; \forall x \in \mathbb{R}^n.
    \end{equation} 
    Hence, by Theorem \ref{prop:ies}, the gradient flow system is globally IES and $x^{\ast}$ is globally exponentially stable.
\end{proof}

\subsection{Proof of Theorem \ref{prop:above-PL-regime}}
\begin{proof}
As in Theorem \ref{prop:global}, we prove that there exists $M(x)$ such that $\mathcal{L}_{M}(x) \le -\nu$ on $\mathbb{R}^n \setminus G$ where $G$ is a small open connected set around $x^*$. Let $g$ be a real-valued function and define $M(x) = Ie^{2g(f(x))}$. By Lemma \ref{lemma:tightness-sysmmetric} and Lemma \ref{lemma:riemmanina-metric-with-lyapunov}, we obtain that a sufficient condition for contractivity to hold in $G$, given a metric $M(x) = e^{2g(f(x))}I$ and $\lambda_{\min}(\nabla^2 f(x)) > -m$, is:
\begin{equation}\label{eq:equivalence-contraction}
\begin{aligned}
 & \mathcal{L}_{M}(x) = -\lambda_{\min}(\nabla^2f(x)) - \frac{\partial g}{\partial f(x)} \|\nabla f(x)\|^2 \le -\beta \\
 & \Rightarrow \frac{\partial g}{\partial f(x)}\|\nabla f(x)\|^2 \ge \beta + m
 \end{aligned}
\end{equation}

Define, as in Theorem \ref{prop:global}, $s_0 = \inf_{x \in \mathbb{R}^n \setminus G} f(x)$,
\begin{equation}
    \psi(s) = \int_{s_0}^{s} \frac{1}{\alpha(\tau - f^*)^2} d\tau,
\end{equation}

and for some $c > 0$, let $g(f(x)) = c \; \psi(f(x)).$ By the Fundamental Theorem of Calculus and $\|\nabla f(x)\| \geq \alpha(f(x) - f^*)$:

\begin{equation}
\frac{\partial g}{\partial f(x)} \|\nabla f(x)\|^2  = 
\frac{c\|\nabla f(x)\|^2}{\alpha(f(x)-f^*)^2} \geq c.
\end{equation}

Choosing $c \ge \beta+m$ ensures \eqref{eq:equivalence-contraction} holds, and by \eqref{eq:strong-PLI}, $g(x)$ is uniformly-bounded above, which implies $M(x)$ is uniformly bounded above and below. As in the proof of Theorem \ref{prop:global}, Lemma \ref{lemma:partition-of-metrics} extends $M$ into a smooth metric $\hat{M}$ that satisfies the contractivity condition in $\mathbb{R}^n$ while keeping the property that $\hat{M}$ is uniformly bounded above and below. By Theorem \ref{prop:ies} the gradient flow system is globally IES and also $x^{\ast}$ is globally exponentially stable.
\end{proof}

\subsection{Generalization of Theorem \ref{prop:concavity-on-annulus}}

\begin{lemma}\label{lemma:forward-invariance-of-inf-contracting-ball}
    Suppose \eqref{eq:ode} has an equilibrium $x^{\ast}$ and is infinitesimally contracting, with rate $\nu$, w.r.t a norm $\| \cdot \|$ on the ball $B_{\epsilon}(x^{\ast}) = \{x \; | \; \|x - x^{\ast}\| < \epsilon \}$. Then $B_{\epsilon}(x^{\ast})$ is forward invariant.
\end{lemma}

\begin{proof}
Suppose there exists $x \in B_{\epsilon}(x^{\ast})$ such that $\varphi(t,x) \notin B_{\epsilon}(x^{\ast})$. Let $\gamma(r)=rx+(1-r)x^{\ast}$ for $r \in [0,1]$, $\psi(t,r)=\varphi(t,\gamma(r))$ and $w(t,r)=\frac{\partial}{\partial r}\psi(t,r)$. By convexity of $B_\epsilon(x^{\ast})$ the compact set $\psi(0,[0,1]) \subset B_\epsilon(x^{\ast})$. Further using that $\varphi(t,x) \notin B_{\epsilon}(x^{\ast})$ and the continuity of solutions there exists
\begin{equation}
    t_e = \inf \{t \ge 0: \exists r \in [0,1] \textbf{ s.t } \psi(t,r) \notin B_{\epsilon}(x^{\ast})\}
\end{equation}
such that $t_e > 0$, hence $\psi(t,r) \in B_{\epsilon}(x^{\ast})$ for $t \in [0,t_e)$ and $r \in [0,1]$. Now differentiating and using the chain rule, $\frac{\partial w}{\partial t}(t,r) = \frac{\partial}{\partial r}g(\psi(t,r)) = \nabla g(\psi(t,r)) w(t,r)$. Hence by Coppel's inequality:
\begin{equation}
    \|w(t,r)\| \le \|w(0,r)\| e^{\int_{0}^{t} \mu(\nabla g(\psi(\tau,r)))\,d\tau}.
\end{equation}
Since \eqref{eq:ode} is infinitesimally contracting in $B_{\epsilon}(x^{\ast})$ with rate $\nu$, for $t \in [0,t_e)$, $\|w(t,r)\|\le e^{-\nu t}\|w(0,r)\|$. $w(0,r)=\gamma'(r)=x-x^{\ast},$ and the Fundamental Theorem of Calculus gives $\psi(t,r)- x^{\ast} = \psi(t,r)- \psi(t,0) = \int_0^r w(t,s)\,ds,$ thus for $t \in [0,t_e)$ and $r \in [0,1]$,
\begin{equation}
\begin{aligned}
& \|\psi(t,r)-x^{\ast}\|
\le \int_0^r \|w(t,s)\|\,ds \\
& \le re^{-\nu t}\|x-x^{\ast}\| < r e^{-\nu t}\epsilon.
\end{aligned}
\end{equation}
By definition of $t_e$ and continuity of solutions there exists some $r \in [0,1]$ such that $\lim_{t \to t_e} \|\psi(t,r)-x^{\ast}\| = \epsilon$, but the above shows $\lim_{t \to t_e} \|\psi(t,r)-x^{\ast}\| < \epsilon$ hence we obtain a contradiction and $B_{\epsilon}(x^{\ast})$ is forward-invariant. 
\end{proof}

\begin{theorem}\label{thm:ies-inf-contracting-compact}
Suppose \eqref{eq:ode} has a unique globally asymptotically stable equilibrium \(x^\ast\) and is infinitesimally contracting, with rate $\nu$ and w.r.t a norm $\|\cdot\|$, everywhere outside a compact set \(K\subset \mathbb{R}^n\) with \(x^\ast\notin K\). Then \eqref{eq:ode} is IES.
\end{theorem}
    
\begin{proof}
Fix \(\epsilon>0\) such that $B_{\epsilon}(x^{\ast}) \subset K^c$. Since \(x^\ast\) is globally asymptotically stable and \(K\) is compact, there exists \(T>0\) such that $\varphi(T,x)\in B_\epsilon(x^\ast),\; \forall x\in K$. By the semigroup property, if \(\varphi(t_0,x)\in K\) for some \(t_0\ge 0\), then $\varphi(t_0+T,x)\in B_\epsilon(x^\ast)$. Moreover, by Lemma \ref{lemma:forward-invariance-of-inf-contracting-ball} $B_\epsilon(x^\ast)$ is forward invariant. Consequently, any trajectory can spend at most a total time $T$ inside $K$: once it enters \(K\), within time \(T\) it reaches $B_\epsilon(x^\ast)$ and then remains there forever. Now let $\gamma(r)=r x_0+(1-r)z_0$ for $r \in [0,1]$, $\psi(t,r)=\varphi(t,\gamma(r))$ and
$w(t,r)=\frac{\partial}{\partial r}\psi(t,r)$. Differentiating and using the chain rule, $\frac{\partial w}{\partial t}(t,r) = \frac{\partial}{\partial r}g(\psi(t,r)) = \nabla g(\psi(t,r)) w(t,r)$. Hence by Coppel's inequality:
\begin{equation}
    \|w(t,r)\| \le \|w(0,r)\| e^{\int_{0}^{t} \mu(\nabla g(\psi(\tau,r)))\,d\tau}.
\end{equation}
Let $M:=\max_{x\in K}\mu(\nabla g(x))$, which is finite by continuity of \(\nabla g\) and compactness of $K$. Since each trajectory \(\psi(t,r)\) spends at most time $T$ in $K$, and $\mu(\nabla g)\le -\nu$ on $K^c$:
\begin{equation}
\int_0^t \mu\!\left(\nabla g(\psi(\tau,r))\right)\,d\tau
\le MT-\nu(t-T)
=
(M+\nu)T-\nu t.
\end{equation}
Therefore, $\|w(t,r)\|\le e^{(M+\nu)T}e^{-\nu t}\|w(0,r)\|.$ $w(0,r)=\gamma'(r)=x_0-z_0,$ and the Fundamental Theorem of Calculus gives $\varphi(t,x_0)- \varphi(t,z_0) = \int_0^1 w(t,r)\,dr,$ thus
\begin{equation}
\begin{aligned}
& \|\varphi(t,x_0)-\varphi(t,z_0)\|
\le
\int_0^1 \|w(t,r)\|\,dr \\
& \le
e^{(M+\nu)T}e^{-\nu t}\|x_0-z_0\|.
\end{aligned}
\end{equation}
\end{proof}

\end{document}